\documentclass[runningheads]{llncs}

\usepackage[T1]{fontenc}
\usepackage{amssymb, amsmath, graphicx, enumerate, enumitem, color} 
\usepackage{tikz}

\usepackage[longnamesfirst,numbers,sort&compress]{natbib}

\usepackage{hyperref}

\usepackage{cleveref}

\newcommand{\bbN}{\mathbb{N}}

\newcommand{\cgC}{\mathcal{C}}

\newcommand{\cgF}{\mathcal{F}}
\newcommand{\cgG}{\mathcal{G}}

\newcommand{\cgS}{\mathcal{S}}

\newcommand{\cgV}{\mathcal{V}}

\newcommand{\tup}{\operatorname{\cgS}}
\newcommand{\otup}{\operatorname{\overline{\cgS}}}

\newcommand{\oS}{\overline{S}}

\let\leq\leqslant
\let\geq\geqslant

\let\epsilon\varepsilon

\let\setminus-

\begin{document}

\title{Graph Reconstruction with Connectivity Queries}
%
%
\author{Kacper Kluk\inst{1} \and Hoang La\inst{2} \and Marta Piecyk\inst{3}}
\authorrunning{K. Kluk et al.}
%
\institute{University of Warsaw, Faculty of Mathematics, Informatics and Mechanics \\ \email{kacper.kluk@mimuw.edu.pl} \and
LISN, Université Paris-Saclay, CNRS, Gif-sur-Yvette, France.\\ \email{hoang.la.research@gmail.com} \and
Warsaw University of Technology, Faculty of Mathematics and Information Science\\ \email{marta.piecyk.dokt@pw.edu.pl}}
\maketitle              
\begin{abstract}
We study a problem of reconstruction of connected graphs where the input gives all subsets of size k that induce a connected subgraph. Originally introduced by Bastide \textit{et al.} (WG 2023) for triples ($k=3$), this problem received comprehensive attention in their work, alongside a study by Qi, who provided a complete characterization of graphs uniquely reconstructible via their connected triples, i.e. no other graphs share the same set of connected triples. Our contribution consists in output-polynomial time algorithms that enumerate every triangle-free graph (resp. every graph with bounded maximum degree) that is consistent with a specified set of connected $k$-sets. Notably, we prove that triangle-free graphs are uniquely reconstructible, while graphs with bounded maximum degree that are consistent with the same $k$-sets share a substantial common structure, differing only locally. We suspect that the problem is NP-hard in general and provide a NP-hardness proof for a variant where the connectivity is specified for only some $k$-sets (with $k$ at least 4).

\end{abstract}

\section{Introduction}


Graph reconstruction primarily refers to the concept popularized by the Graph Reconstruction Conjecture by Kelly and Ulam~\cite{Kelly1942,Ulam1960}. This conjecture states that any graph with a minimum of three vertices can be uniquely reconstructed, up to isomorphism, using the multiset of subgraphs resulting from the removal of exactly one vertex. However, in this context, we delve into the reconstruction of labeled graphs, where we lack explicit knowledge about the structures of subgraphs of order $n-1$, as proposed in the Graph Reconstruction Conjecture. Instead, our focus lies on reconstructing graphs based solely on the connectivity of their subgraphs. Specifically, we address the task of reconstructing graphs by querying if a subgraph of order $k$ is connected, for some fixed integer $k\geq 2$. This variant of graph reconstruction was recently introduced by Bastide, Cook, Erickson, Groenland, Kreveld, Mannens, and Vermeulen~\cite{BCEGKMV2023}, initially for triples, and more broadly for sets of vertices of size $k$, which we term $k$-sets for brevity.

Some problems lend themselves to straightforward solutions when access to answers for all possible queries is available, such as graph reconstruction using a distance oracle~\cite{KMZ2018}. In such scenarios, our focus shifts towards optimizing the number of queries utilized. In graph reconstruction via connectivity queries on $k$-sets, it is evident that for $k=2$, we are effectively querying the edges of the graphs. However, for $k\geq 3$, even identifying a graph consistent with the answers to all queries becomes a nontrivial task.

Graph reconstruction inherently suggests the existence of a target graph we aim to rebuild. Yet, within this context, we may encounter arbitrary connected and disconnected $k$-sets. The questions regarding the existence and uniqueness of graphs satisfying these queries become significantly more intriguing. This problem bears potential applications in computational networks where concealing its exact structure is imperative while still being able to respond to local queries about potential data transfers to a sub-network, which would require connectivity, is necessary.

\textbf{Hardness of the general case.} For the case where $k=3$, Bastide et al. showed an enumeration algorithm applicable to general graphs, leveraging a reduction of the problem to 2-SAT. Moreover, Qi~\cite{Qi2023} provided a characterization of graphs that are uniquely reconstructible from connected triples. However, extending this reduction to $k\geq 4$ merely leads to the problem being reduced to 3-SAT or simply SAT. Moreover, we suspect that for $k\geq 4$, the task of identifying a graph consistent with a given set of connected $k$-sets may be NP-hard in general. To substantiate this suspicion, we show NP-hardness for a modified version of this problem, wherein not every connected and disconnected $k$-set is given. We point out, that such a version of the problem is still tractable for $k=3$, as the reduction to $2$-SAT can handle with an uncomplete information.

\textbf{Motivation.} Shifting our focus away from the general scenario, what if we possess additional information regarding the target graph or specific graph classes? Our attention naturally gravitates towards classes wherein existence can be efficiently verified or uniqueness can be established structurally. For instance, what if we know that our target graph is triangle-free or has bounded maximum degree? In what follows, we delve into the rationale and intuition behind our findings.

The graphs that we aim to rebuild are all labeled, finite, and simple. It is essential to note that this form of reconstruction does not always yield a unique labeled graph: a path on vertices $a$, $b$, $c$, $d$ has the same connected triples as the isomorphic path on $a$, $c$, $b$, $d$. More generally, we do not always obtain isomorphic graphs. For example, a clique and a clique lacking a matching have identical connected $k$-sets. We point out that from the connected $k$-sets, we are able to determine the vertex sets of all connected components of size at least $k$, but we cannot distinguish connected components of size at most $k-1$. Therefore, we will always assume that our graphs are connected and reasonably large.

Here, we direct your attention to two specific examples illustrating the challenge of reconstructing graphs solely through connectivity. The first example revolves around the existence of edges between twin vertices. Twin vertices share precisely the same neighbors, excluding themselves. Consequently, when examining their connected $k$-sets, we cannot discern between a graph where they share an edge and one where they do not.

The second example involves a star or, more generally, a graph with a universal vertex. Consider $G'$ a disconnected graph, where each connected component has order of at most $k-1$. Now, let $G$ be $G'$ with an additional universal vertex. Analogous to the previous example, distinguishing between $G$ and a star of degree $|V(G')|$ based on connected $k$-sets is impossible.

\textbf{Our results.} In both of the previous examples, with the additional assumption of triangle-freeness, we are left with only one potential outcome. Indeed, not only do we prove that enumerating connected triangle-free graphs that are consistent with some given connected $k$-sets can be done efficiently, but we also show that when such a graph exists, there are no alternative connected triangle-free graphs sharing the same $k$-sets. More formally, we prove the following statements.

\begin{theorem}
Let $k$ be an integer with $k \geq 2$. For a complete set of connected and disconnected $k$-sets on $V$, we can enumerate every connected triangle-free graph on $V$ consistent with such sets in polynomial time in $|V|$.
\end{theorem}

\begin{theorem}
Let $k$ be an integer with $k\geq 3$. If $G$ is a large enough triangle-free connected graph, then $G$ is uniquely reconstructible in the class of triangle-free graphs. 
\end{theorem}

In the second example, the presence of the universal vertex, ensuring the graph's connectivity, opens the door to a multitude of different structures that can be concealed within a set of vertices spanning (almost) the entire graph. Intuitively, bounding the maximum degree of the graph restricts such variability to only sets of constant size, at most the maximum degree. While this constraint does not guarantee uniqueness -- in fact, quite the opposite; we may still have an exponential number of possible subquartic graphs consistent with the same connected $k$-sets (a path $u_1\dots u_n$ where each $u_i$ has two extra neighbors $v_i$ and $w_i$ that are potentially adjacent) -- we anticipate these graphs to exhibit a substantial common structure (the common path $u_1\dots u_n$ with the extra neighbors) and differing only locally (the potential edge between $v_i$ and $w_i$). Indeed, we confirm these intuitions and prove the following.

\begin{theorem}
    Let $k$ and $d$ be integers such that $k\geq 2$ and $d\geq 1$. For a complete set of connected and disconnected $k$-sets on $V$, we can enumerate every connected graph of maximum degree $d$ consistent with such $k$-sets in polynomial time in $|V|$ and the size of the output.
\end{theorem}

This theorem comes from a more technical result in which we introduce the concept of a ``skeleton'' which captures this large common structure shared among consistent graphs and explicitly shows their local differences.

\textbf{Outline.} In \Cref{sec:prelim}, we introduce basic notations as well as notions pertinent to our reconstruction problem. \Cref{sec:alg} describes our enumeration algorithms for both triangle-free and bounded maximum degree graphs. Following this, \Cref{sec:uniqueness} presents the structural result on uniquely reconstructible triangle-free graphs. 
Finally, in the Appendix, we provide the NP-hardness proof of the variant of this problem, the proofs of the statements marked with (*), as well as some extra technical propositions required for our results.

\section{Preliminaries} \label{sec:prelim}
For a positive integer $n$, by $[n]$ we denote the set $\{1,\ldots,n\}$. We use standard graph notations, i.e., a graph is a pair $G=(V,E)$, where $V$ is the set of vertices and $E\subseteq \binom{V}{2}$ is the set of edges of $G$. By $\deg_G(v)$ we denote the degree of $v$ in $G$ and by $N_G(v)$ we denote the set of neighbors of $v$ in $G$. We denote the closed neighborhood of $v$ by $N_G[v]=N_G(v)\cup\{v\}$. Similarly, for $S\subseteq V$, the closed neighborhood of $S$ is $N_G[S]=\bigcup_{v\in S}N_G[v]$ and the open neighborhood of $S$ is $N_G(S)=N_G[S]-S$. For a set of vertices $S$, by $G[S]$, we denote the subgraph of $G$ induced by $S$. If $G$ is clear from the context, we omit the subscript $G$.

\textbf{Partial graphs.} A \emph{partial graph} is a tuple $H=(V,E,E_N,E_U)$, where $V$ is the vertex set, $E,E_N,E_U\subseteq \binom{V}{2}$, and $E\cup E_N\cup E_U= \binom{V}{2}$. The sets $E,E_N,E_U$ denote respectively (known) edges, (known) non-edges and unknown edges. In particular, a graph $G=(V,E)$ can be seen as a partial graph $(V,E,\binom{V}{2}\setminus E,\emptyset)$. 

For a partial graph $H=(V,E,E_N,E_U)$, we say that a partial graph $H'$ on $V$ is a \emph{partial supergraph} $H'$ of $H$ if $E\subseteq E(H')$, $E_N\subseteq E_N(H')$, and $E_U(H')\subseteq E_U$.
We say that a graph $G$ is a \emph{supergraph} of $H$ if $E\subseteq E(G)$ and $E_N\subseteq \binom{V}{2} \setminus E$. In other words we can obtain $G$ (resp. $H'$) from $H$ by deciding for every unknown edge (resp. for some subset of unknown edges) if it is an edge or a non-edge.

For $H=(V,E,E_N,E_U)$, let $G$ be the graph $(V,E)$. We say that a set of vertices induces a connected subgraph in $H$ if it induces a connected subgraph in $G$. The open neighborhood of $S$ in $H$ is $N_H(S)=N_G(S)$.

\textbf{Connected $k$-sets.} Let $G=(V,E)$ be a graph and let $k\in \bbN$. By $\tup_k(G)$ we denote the set of all $k$-subsets of $V$ that induce a connected subgraph of $G$. Similarly, by $\otup_k(G)$ we denote the set of all $k$-subsets of $V$ that induce a disconnected subgraph of $G$. 

\textbf{Consistent graphs.} We say that \emph{a graph $G$ is consistent with $\tup_k$ and $\otup_k$} when $\tup_k(G)=\tup_k$ and $\otup_k(G)=\otup_k$. Let $H=(V,E,E_N,E_U)$ be a partial graph. We say that \emph{$H$ is consistent with $\tup_k$ and $\otup_k$} if there exists an assignment of unknown edges ($E_U$) of $H$ to obtain a graph $G$ consistent with $\tup_k$ and $\otup_k$.

\textbf{Uniquely reconstructible graphs.} We say that a graph $G$ is \emph{uniquely reconstructible in the class of graphs $\cgG$} if it is the only graph in $\cgG$ consistent with $\tup_k(G)$ and $\otup_k(G)$. We omit $\cgG$ when we consider all graphs.

\section{Enumerating consistent graphs} \label{sec:alg}

We start by showing \Cref{lem:alg_layer_single,lem:alg_layers} which allows us to find a breadth-first-search partition of the vertices of $V$ starting from some small initial set $T$ to obtain a partial graph that is consistent with $\tup_k$ and $\otup_k$. Our reconstruction algorithms consist in guessing this small set $T$ and its neighborhood, building the ``layering'' corresponding to this partition, and filling in the rest depending on the target class of graphs (triangle-free or bounded maximum degree).

\begin{lemma}\label{lem:alg_layer_single}
    Let $\tup_k, \otup_k$ be a complete set of $k$-sets on $V$. Let $H$ be a partial graph on $V$ and $T$ be a subset of $V$ such that:
    \begin{enumerate}
        \item $T$ is connected in $H$, 
        \item $|T|\geq k-1$,
        \item and $N_H(T)$ is non-empty.
    \end{enumerate}
    We can find $H'$ a partial supergraph of $H$ using at most $|V|\cdot |N_H(T)|$ queries such that:
    \begin{enumerate}
        \item for every $x\in V-N_H[T]$ and $y\in N_H(T)$, $xy\notin E_U(H')$,
        \item and for every supergraph $G$ of $H$ such that $N_G(T)=N_H(T)$ and $G$ is consistent with $\tup_k, \otup_k$, $G$ is a supergraph of $H'$.
    \end{enumerate}
\end{lemma}

\begin{proof}
    For every vertex $y \in N_H(T)$ let $C_y$ denote an arbitrary connected subset of $k - 2$ vertices from $T$ adjacent to $y$ in $H$.
    For every pair $y \in N_H(T), x \in V- N_H[T]$ such that $xy \in E_U(H)$, we define the $k$-set $s_{x, y} = \{x, y\} \cup C_y$. If $s_{x,y} \in \tup_k$, we put $xy$ to be an edge in $H'$. Otherwise we put $xy$ to be a non-edge.
    All other pairs in $H'$ remain defined as in $H$.
    
    It is clear that $H'$ satisfies (i).
    Let $G$ denote an arbitrary supergraph of $H$ consistent with $\tup_k, \otup_k$ such that $N_G(T) = N_H(T)$. Pick any pair $xy \in E_U(H) - E_U(H')$. By the definition of $H'$, we have $x \in V - N_H[T]$ and $y \in N_H(T)$. Since $C_y \cup \{y\}$ was connected in $H$, it is also connected in $G$. Moreover, since $N_G(T) = N_H(T)$, there are no edges between $x$ and $T$ in $G$. Therefore, $s_{x, y} \in \tup_k$ if and only if $xy \in E(G)$. This implies that $G$ is a supergraph of $H'$, which finishes the proof.\qed
\end{proof}

\begin{lemma}[\textbf{Layering}]\label{lem:alg_layers}
    Let $\tup_k, \otup_k, H, T$ be defined as in \Cref{lem:alg_layer_single}.
    We can find $H'$ a partial supergraph of $H$ together with the partition $T = L_0, L_1, \dots, L_\ell$ of $V$ using at most $|V|^2$ queries such that:
    \begin{enumerate}
        \item for every $0 \leq i < j \leq \ell$ except $i = 0, j = 1$, and every $x \in L_i, y \in L_j$, $xy \notin E_U(H')$,
        \item $L_i$ are the vertices at distance exactly $i$ from $T$ in $H'$,
        \item and for every supergraph $G$ of $H$ such that $N_G(T)=N_H(T)$ and $G$ is consistent with $\tup_k, \otup_k$, $G$ is a supergraph of $H'$.
    \end{enumerate}
    We also refer to $(L_i)_{i=0}^\ell$ as a layering.
\end{lemma}

\begin{proof}
    Let $H_0$ be $H$ with all pairs $xy$ for $x \in T$, $y \in V - N_H[T]$ put into $E_N(H_0)$. Let $L_0 = \bar{L}_0 = T$. For $k \geq 1$ we define recursively $L_k = N_{H_{k - 1}}(\bar{L}_{k - 1}), \bar{L}_k = \bar{L}_{k - 1} \cup L_k$ and $H_k$ to be the graph obtained by applying \Cref{lem:alg_layer_single} to $H_{k - 1}$ and the set $\bar{L}_{k - 1}$ in place of $T$. We stop once $L_k = V$ and put $\ell = k$ and $H' = H_k$. Note that $H_j$ is a supergraph of $H_i$ for every $0 \leq i < j \leq \ell$. Now we will prove that such choice of $L_i$ and $H'$ satisfies the conditions of the lemma.

    \begin{claim}
        The first point of the lemma statement holds.
    \end{claim}
    \begin{proof}
        Fix any such pair $i, j$. Assume $i = 0$. Then $x \in L_0 = T$ and $y \in V - \bar{L}_1 = V - N_{H_0}[T] = V - N_H[T]$. By the definition of $H_0$, we have $xy \notin E_U(H_0)$, hence $xy \notin E_U(H')$.
        Assume $i > 0$. Then $x \in L_i = N_{H_{i - 1}}(\bar{L}_{i-1})$ and $y \in V - \bar{L}_{i} = V - N_{H_{i - 1}}[\bar{L}_{i-1}]$. By \Cref{lem:alg_layer_single} and the definition of $H_i$, $xy \notin E_U(H_i)$, hence $xy \notin E_U(H')$.
    \end{proof}
    
    \begin{claim}
        Fix $0 \leq i, j \leq \ell, i + 2 \leq j$ and vertices $x \in L_i, y \in L_j$. Then $xy \in E_N(H')$.
    \end{claim}
    \begin{proof}
        If $i = 0$, then $x \in T, y \in V - \bar{L}_1$ and by definition of $H_0$, $xy \in E_N(H_0) \subseteq E_N(H')$.
        Assume $i > 0$. By \Cref{lem:alg_layer_single} and the definition of $H_i$, $xy \notin E_U(H_i)$. If $xy \in E(H_i)$, then $y \in N_{H_i}(x)$, hence $y \in N_{H_i}[\bar{L}_i] = \bar{L}_{i + 1}$, a contradiction. Therefore, $xy \in E_N(H_i) \subseteq E_N(H')$.
    \end{proof}
    
    \begin{claim}
        The second point of the lemma statement holds.
    \end{claim}
    \begin{proof}
        By the definition of $L_i$ for $i > 0$, $L_i \subseteq N_{H_{i - 1}}(L_{i - 1}) \subseteq N_{H'}(L_{i - 1})$, hence every vertex of $i$-th layer has a neighbour in $(i-1)$-th layer in $H'$. Fix any $i$ and $x \in L_i$. If $i = 0$, then $x \in T$. Otherwise, we can find a path from $x$ to $T$ of length $i$ in $H'$, by starting in $x$ and at every step moving to an arbitrary neighbour in previous layer. On the other hand, any path of length shorter than $i$ would contain a consecutive pair $x, y$ such that $x \in L_{j_1}, y \in L_{j_2}$ and $|j_1 - j_2| \geq 2$, which contradicts the previous claim.
    \end{proof}
    
    \begin{claim}
        The third point of the lemma statement holds.
    \end{claim}
    \begin{proof}
        From the definition of $H_0$, $N_H(T) = N_{H_0}(T)$ and all edges between $T$ and $V - N_H[T]$ are known in $H_0$, hence $N_{H_i}(T) = N_H(T)$ for all $0 \leq i \leq \ell$. Fix any supergraph $G$ of $H$ consistent with $\tup_k, \otup_k$ such that $N_G(T) = N_H(T)$.
        Clearly, $G$ is a supergraph of $H_0$. Assume that $G$ is not a supergraph of some $H_i$ for such smallest possible choice of $i$. Then, $G$ is a supergraph of $H_{i - 1}$ and we have $N_G(T) = N_{H_{i - 1}}(T)$, hence by \Cref{lem:alg_layer_single} and the definition of $H_i$, $G$ is a supergraph of $H_i$, which is a contradiction.
    \end{proof}
    This finishes the proof of the lemma.\qed
\end{proof}

\subsection{Triangle-free graphs}\label{alg:triangle-free}

The following lemma shows an unique reconstruction of a triangle-free graph given that $T$ and its neighborhood is known.

\begin{lemma}\label{lem:k3_free_alg_finish}
    Let $\tup_k, \otup_k, H, T$ be defined as in \Cref{lem:alg_layer_single}. Moreover, let all edges of $H[N_H[T]]$ be known. Then there exist at most one triangle-free supergraph $G$ of $H$ consistent with $\tup_k, \otup_k$ such that $N_G(T) = N_H(T)$. Moreover, we can find this graph or determine it doesn't exist in polynomial time.
\end{lemma}

\begin{proof}
    Let $G$ denote any supergraph of $H$ consistent with $\tup_k, \otup_k$ such that $N_G(T) = N_H(T)$.
    
    Using \Cref{lem:alg_layers}, we get a partial graph $H'$ and partition $(L_i)^{\ell}_{i=0}$ of $V$, where $G$ is a supergraph of $H'$. Moreover, for every unknown edge $vw$ of $H'$, we have $v, w \in L_i$ for some $2 \leq i \leq \ell$.
    
    Fix any $2 \leq i \leq \ell$ and a pair of vertices $v, w \in L_i$. If $v, w$ have a common neighbour in $L_{i - 1}$ in $H'$, then $vw \not\in E(G)$. Otherwise, pick any $x_v \in L_{i - 1} \cup N_{H'}(v)$ and any $(k-3)$-subset $Z$ of $\bigcup_{j=0}^{i - 2} L_j$ connected and adjacent to $x_v$ in $H'$. We look at the $k$-set consisting of $v, w, x_v$ and $Z$. $v, x_v$ and $Z$ are connected in $G$ and $w$ is non-adjacent to both $x_v$ and $Z$, hence such $k$-set is connected if and only if $vw \in E(G)$.

    This proves that the remaining edges of any supergraph of $H$ which satisfies the assumptions are uniquely determined. \qed
\end{proof}

For triangle-free graphs, we will choose this $T$ to be the closed neighborhood of a vertex $u$ with large enough degree when it exists since what follows show that we can easily determine its neighborhood and vertices at distance 2 (which would correspond to $N[T]$) just by knowing a small part of $N(u)$.

\begin{proposition}\label{l:k3free_kernel}
    Fix $k \geq 2$. Let $\tup_k, \otup_k$ be a complete set of $k$-sets on $V$. Given vertices $u, x_1, x_2, \dots, x_{2k - 4}$, we can find subsets $X, Y \subseteq V$ and a partial graph $H$ on $V$ such that:
    \begin{enumerate}
        \item $X = N_H(u)$ and $Y = N_H(\{u\} \cup X)$,
        \item for every $s \in \{u\} \cup X \cup Y$ and $t \in V$, $st \notin E_U(H)$,
        \item and for every triangle-free graph $G$ consistent with $\tup_k, \otup_k$ such that\\ $x_1, \dots, x_{2k - 4} \in N_G(u)$, $G$ is a supergraph of $H$.
    \end{enumerate}
\end{proposition}

\begin{proof}
    Let $G$ denote any graph satisfying the assumptions in (iii). We will construct $H$ step by step and prove that all added edges and non-edges are consistent with $G$.

    Let $Z = \{x_1, \dots, x_{2k - 4}\}$. Fix $v \notin \{u\} \cup Z$.
    We query a $k$-set $s_{v, Z'} = \{v, u\} \cup Z'$ for every $(k-2)$-subset $Z'$ of $Z$. We query a $k$-set $t_{v, Z''} = \{v, u\} \cup Z''$ for every $(k-1)$-subset $Z''$ of $Z$.
    If $v \in N_G(u)$, all $s_{v, Z'}$ are connected. Additionally, since $G$ is triangle-free, $N_G(u)$ is an independent set and all $t_{v, Z''}$ are disconnected.
    On the other hand, if $v \not\in N_G(u)$, then either $v$ has some $k - 1$ neighbours in $Z$ and the corresponding $t_{v, Z''}$ is connected or $v$ has some $k - 2$ non-neighbours in $Z$ and the corresponding $s_{v, Z'}$ is disconnected.
    Therefore, we can define $X$ to be $Z$ plus all such $v$ for which all $s_{v, Z'}$ are in $\tup_k$ and all $t_{v, Z''}$ are in $\otup_k$. For every $v \neq u$, we put $uv \in E(H)$ if $v \in X$, and we put $uv \in E_N(H)$ otherwise.
    
    Fix $v \notin \{u\} \cup X$.
    We query a $k$-set $s_{v, X'} = \{v, u\} \cup X'$ for every $(k-2)$-subset $X'$ of $X$. We query a $k$-set $t_{v, X''} = \{v, u\} \cup X''$ for every $(k-1)$-subset $X''$ of $X$.
    The vertex $v$ belongs to $N_G(\{u\} \cup X)$ if and only if at least one $s_{v, X'}$ is connected, hence we put $Y$ to be the set of all such vertices $v$.
    Again, $v$ has either at least $k - 1$ neighbours in $X$ or $k - 2$ non-neighbours in $X$. In the first case, the union of $t_{v, X''} \in \tup_k$ is exactly $\{v\} \cup (N_G(v) \cap X)$. If no $t_{v, X''}$ is in $\tup_k$, then we are in the second case and the union of $s_{v, X'}$ is exactly $\{u, v\} \cup (X - N_G(v))$.
    This means that the edges between $X$ and $V - \{u\} - Y$ can be uniquely determined in $G$. We define them accordingly in $H$.
    
    Finally, fix a pair of vertices $v, w \in Y$. Since $G$ is triangle-free, if $v$ and $w$ have a common neighbour in $X$, then $vw \notin E(G)$. Assume their neighbourhoods in $X$ are disjoint. W.l.o.g. we can assume that $v$ has at least $k - 3$ non-neighbours. Pick an arbitrary neighbour $x_w$ of $w$ in $X$. We query the $k$-set consisting of $u, v, w, x_w$ and some additional $k - 4$ vertices of $X - \{x_w\} - N_G(v)$. Since $v$ is non-adjacent to $u$ and $X - N_G(v)$, such tuple is connected if and only if $vw \in E(G)$. This means that the edges of $G[Y]$ can be uniquely determined in $G$. We define them accordingly in $H$.
    It is easy to see that $X, Y$ and $H$ defined in such way satisfies all the properties.
\end{proof}

\begin{lemma}\label{l:k3free_large_deg_graph}
    Fix $k \geq 2$. Let $\tup_k, \otup_k$ be a complete set of $k$-sets on $V$. Given vertices $u, x_1, x_2, \dots, x_{2k - 4}$, there exists at most one triangle-free graph $G$ consistent with $\tup_k, \otup_k$ such that $\{x_1, \dots, x_{2k-4}\} \in N_G(u)$. Moreover, we can find this graph or determine it doesn't exist in polynomial time.
\end{lemma}

\begin{proof}
     Let $G$ denote any triangle-free graph consistent with $\tup_k, \otup_k$ such that\\ $x_1, \dots, x_{2k-4} \in N_G(u)$.

     Applying \Cref{l:k3free_kernel}, 
     we obtain a partial graph $H$ and subsets $X, Y \subseteq V$ such that $G$ is a supergraph of $H$, $N_G(u) = N_H(u) = X$ and $N_G(\{u\} \cup X) = Y$. Moreover, all edges of $H[\{u\} \cup X \cup Y]$ are known.
    
     Let $T = \{u\} \cup X$. We will show that the assumptions of \Cref{lem:k3_free_alg_finish} are satisfied by $H$ and $T$. Clearly, $T$ is connected in $H$ and of size at least $2k - 3 \geq k - 1$ and $N_H(T) = Y$ is non-empty. Finally, since $\{u\} \cup X \cup Y = T \cup N_H(T)$, all edges of $H[T \cup N_H(T)]$ are known. This allows us to use \Cref{lem:k3_free_alg_finish}, which finishes the proof.
 \end{proof}

Finally, we assemble the different components of our algorithm to enumerate every connected triangle-free graph that is consistent with $\tup_k$ and $\otup_k$. Moreover, as we will see in \Cref{sec:unique-triangle-free}, for large enough graphs, only one such graph can exist.

\begin{theorem}\label{thm:alg-triangle-free}
    Let $k$ be an integer with $k \geq 2$. For a complete set of $k$-sets $\tup_k,\otup_k$ on $V$, we can enumerate every connected triangle-free graph on $V$ consistent with $\tup_k,\otup_k$ in polynomial time.
\end{theorem}

\begin{proof}
    Let $G$ denote an arbitrary triangle-free graph consistent with $\tup_k, \otup_k$. We consider two cases: $G$ contains a vertex of degree at least $2k - 4$ or not.
    
    In the first case, we go over all candidates $u \in V$ of vertex of such degree and for every subset $\{x_1, ..., x_{2k-4}\} \subseteq V$ of candidate neighbours of $u$. For every selection of $u$ and $x_i$, we apply \Cref{l:k3free_large_deg_graph} to find a potential solution $G$.

    In the second case, we pick any connected $k$-set $T \in \tup_K(S)$. By our degree assumption, the size of $N_G(T)$ in any solution $G$ is at most $k(2k - 4)$. We go over all candidate sets $N$ of size at most $k(2k - 4)$ for this neighbourhood. For fixed $T$ and $N$, we go over all possible graphs $F$ on $T \cup N$ for which $N_F(T) = N$ (since $|T \cup N| = O(1)$, there is a constant number of such graphs). Now we put $H$ to be a partial graph on $V$ with the edges of $H[T \cup N]$ put as in $F$. Now we use \Cref{lem:k3_free_alg_finish} to find a potential solution $G$.

    We are left with a polynomial-size list of potential solutions. Finally, we filter out all solutions that do not satisfy all tuples $\tup_k, \otup_k$.\qed
\end{proof}

\subsection{Graphs with bounded maximum degree}\label{alg:bounded-degree}

In this subsection, we introduce the notion of a skeleton. A skeleton is a partial graph where the unknown edges are contained in some disjoint small sets. The idea is that using a skeleton we can describe a set of graphs in which the differences between its members only occur in small sets of constant size which do not interact with one another. Formally, a skeleton is defined as follows.

\begin{definition}
    Let a \emph{skeleton} be a partial graph $H$ with a family of disjoint vertex subsets $V_1,V_2,\dots,V_m$ and a family of collections $\cgC_1,\cgC_2,\dots,\cgC_m$ such that $E_U(H)=\bigcup_{i\in [\ell]}\binom{V_i}{2}$ and $\cgC_i$ is a collection of graphs on $V_i$ for every $1\leq i\leq m$. The \emph{width} of a skeleton is $\max_{i\in[m]}|V_i|$. A graph $G$ is a \emph{completion} of a skeleton if there exists graphs $C_1\in \cgC_1,\dots,C_m\in \cgC_m$ such that $G$ is a supergraph of $H$ and $G[V_i]=C_i$ for every $1\leq i\leq m$.
\end{definition}

Using this concept, we can capture the common structure present in connected graphs of bounded maximum degree that share the same $k$-sets even though there might be exponentially many of them.

\begin{theorem}\label{thm:bounded-degree}
    Let $k$ and $d$ be integers such that $k\geq 2$ and $d\geq 1$. For a complete set of $k$-sets $\tup_k,\otup_k$ on $V$, in polynomial time, we can find a constant-size (depending on $k$ and $d$) family $\cgF$ of skeletons of width at most $d$ such that a connected graph of maximum degree $d$ is consistent with $\tup_k,\otup_k$ if and only if it is a completion of a skeleton of $\cgF$.
\end{theorem}

This implies that we can enumerate every connected graph of maximum degree $d$ that is consistent with our given $k$-sets in a reasonable time. Moreover, the difference between these graphs are very local since any combination of graphs in $\cgC_1,\dots,\cgC_m$ will produce a desired consistent graph.

\begin{corollary}
    Let $k$ and $d$ be integers such that $k\geq 2$ and $d\geq 1$. For a complete set of $k$-sets $S_k,\oS_k$ on $V$, we can enumerate every connected graph of maximum degree $d$ consistent with $S_k,\oS_k$ in polynomial time in $|V|$ and the size of the output.
\end{corollary}

Before proving the main result of this section, we introduce a notion of importance for unknown edges of a $k$-set $S$ in the current partial graph $H$ which are edges whose presence can influence the connectivity of every supergraph of $H[S]$.

\begin{definition}[Importance of an unknown edge]
Let $S$ be a $k$-set on $V$, let $H$ be a partial graph on $V$, and let $u,v\in S$ such that $uv\in E_U(H)$. We say that \emph{$uv$ is not important for $S$ in $H$} if for every supergraph $G_S$ of $H[S]$, $G_S+uv$ is connected if and only if $G_S-uv$ is connected. Otherwise, we say that \emph{$uv$ is important for $S$ in $H$}.
\end{definition}

The fact that an unknown edge is not important for some $k$-set $S$ in $H$ gives us some specific structures in $H[S]$.

\begin{proposition}\label{prop:unimportant}
Let $S$ be a $k$-set, let $K_S$ be the complete graph on vertices of $S$, let $H$ be a partial graph on the vertices of $S$, and let $u,v\in S$ such that $uv\in E_U(H)$. Then, $(a)$ $uv$ is not important for $S$ in $H$ if and only if 
$(b)$ there exists a path on vertices of $S$ and edges of $E(H)$ between $u$ and $v$ or 
$(c)$ $K_S-E_N(H)$ is disconnected.
\end{proposition}

\begin{proof}
 First, we prove that $(c)$ implies $(a)$. Suppose that $K_S-E_N(H)$ is disconnected. Every supergraph $G_S$ of $H[S]$ must be disconnected. Therefore, $G_S+uv$ and $G_S-uv$ are also disconnected. Hence, $uv$ is not important for $S$ in $H$.

Now, we prove that $(b)$ implies $(a)$. If there exists a path in $S$ between $u$ and $v$ using known edges, i.e. edges of $E(H)$, then every supergraph $G_S$ of $H[S]$ will contain such path. Removing or adding an edge between two vertices that are connected by another path cannot change the connectivity of a graph. Therefore, $G_S+uv$ and $G_S-uv$ must have the same connectivity. Hence, $uv$ is not important for $S$ in $H$.

Finally, we show that if $(b)$ and $(c)$ do not hold, then $(a)$ also cannot hold, i.e. $uv$ is an important edge for $S$ in $H$. To show that $uv$ is important, we will build a supergraph $G_S$ of $H[S]$ such that $G_S+uv$ is connected and $G_S-uv$ is disconnected. First, start with $G'_S=K_S-E_N(H)$ which is a connected supergraph of $H[S]$ by assumption. Since there are no paths between $u$ and $v$ in $S$ using only known edges of $H$, we must have a path in $S$ between $u$ and $v$ using both known and unknown edges, i.e. edges of $E(H)$ and $E_U(H)$. Consider such a path $P$ and an edge $e_P$ on $P$ that is unknown in $H$. Recall that $E(G'_S)=E(H[S])\cup E_U(H[S])$. We can thus remove $e_P$ from $G'_S$. If $G'_S$ is still connected, then we repeat this process by choosing another path between $u$ and $v$ in $G'_S$ and removing another edge of $E_U(H[S])$ from this path until $G'_S$ is disconnected. This process must end because at worst we remove every path between $u$ and $v$ in $G'_S$. We call $G_S$ the resulting graph. By construction, $G_S-uv=G_S$ is disconnected. Let us show that $G_S+uv$ is connected. Let $wx$ be the last edge removed by this process and consider two vertices $y$ and $z$ in $G_S$ with no paths between them. Since $G_S+wx$ is connected, there exists a path $P_{wx}$ from $y$ to $z$ going through $wx$. By construction, $wx$ exists on a path in $G_S+wx$ between $u$ and $v$. Therefore, there exists a path $P_{uv}$ between $w$ and $x$ going through $uv$ in $G_S+uv$. Thus, we have a path between $y$ and $z$ which consists of $P_{wx}$ where $wx$ is replaced by $P_{uv}$ in $G_S+uv$. This shows that $G_S+uv$ is connected and concludes our proof.
\end{proof}

We are ready to prove \Cref{thm:bounded-degree}.
\begin{proof}[\Cref{thm:bounded-degree}]
Let $\tup_k$ and $\otup_k$ be the $k$-sets on $V$ and let $n=|V|$. To prove~\Cref{thm:bounded-degree}, we describe an algorithm that produces the desired skeletons of width $d$ in time polynomial in $n$.
\begin{enumerate}
    \item Let $T\in\tup_k$. Determine the neighborhood of $T$. Enumerate all possible connected graphs with maximum degree $d$ on vertices of the closed neighborhood of $T$ to obtain a family of partial graphs. Proceed with the next steps for each partial graph. 
    \item For each partial graph $H$, obtain a layering $(L_i)^{\ell}_{i=0}$ with \Cref{lem:alg_layers} from $\tup_k$, $\otup_k$, $H$, and $T$. For every $2\leq i\leq \ell$, let $\cgV_i$ be the family of maximal sets of vertices in layer $L_i$ with the same neighborhood in the previous layer $L_{i-1}$. Let $\{V_1,\dots,V_m\}=\bigcup^\ell_{i=2}\cgV_i$ be all such sets for all layers. By definition, these sets have size at most the maximum degree of $G$ which is $d$. Now, determine every unknown edge that is not contained in a $V_j$ ($1\leq j\leq m$).
    \item For every $1\leq j\leq m$, we determine every unknown edge $uv$ in $V_j$ such that $u$ (resp. $v$) is contained in a connected subgraph $G'$ (of the current partial graph) of order $k-1$, $v\notin V(G')$ (resp. $u\notin V(G')$), and $v$ (resp. $u$) is not adjacent to any vertex in $G'-u$ (resp. $G'-v$), i.e $vw\in E_N$ (resp. $uw$) for every $w\in V(G'-u)$ (resp. $w\in V(G'-v)$). Let the resulting partial graph be $H^*$. 
    \item For each $1\leq j\leq m$, we enumerate every possible supergraph of $H^*[V_j]$ and keep the ones that ``satisfy the $k$-sets'' to obtain the family $\cgC_j$.
\end{enumerate}

Now, we describe the different steps in more details and prove that the completions of each skeleton consisting of $H^*$, $V_1,\dots,V_m$, $\cgC_1,\dots,\cgC_m$ are exactly the connected graphs with maximum degree $d$ consistent with $\tup_k$ and $\otup_k$.

An invariant that we keep throughout our algorithm is that we will generate many partial graphs such that the set of supergraphs of these partial graphs contain every graph that is consistent with $\tup_k$ and $\otup_k$. If at some point the (known) edges and non-edges of a partial graph contradicts a given $k$-set, then this partial graph is eliminated from the family of current partial graphs. 

\textbf{Step 1.} Choose an arbitrary connected $k$-set $T$ from $\tup_k$. We can determine the (open) neighborhood of $T$ by checking for each vertex $v$ in $V-T$ if there exists $u\in T$ such that $T-u+v$ is in $\tup_k$. In this way, we obtain every vertex in the closed neighborhood of $T$. This is done in time $O(kn)$. Since we are looking for graphs with maximum degree at most $d$, we have at most $2^{d^2k^2}$ possible graphs on the closed neighborhood of $T$: more precisely, we have at most $dk$ vertices in the open neighborhood of $T$ so $dk+k$ vertices and $\binom{dk+k}{2}$ unknown edges in total; thus, $2^{\binom{dk+k}{2}}$ possible graphs. For each of these graphs, we can check in time $\binom{dk+k}{k}$ if they are consistent with the $k$-sets on the closed neighborhood of $T$. This gives us a family of partial graphs where for each partial graph $H$, the known edges are exactly the ones inside $N_H[T]$. For fixed $k$ and $d$, Step 1 can be done in time $O(n)$.

\textbf{Step 2.} For each consistent partial graph $H$ from Step 1, we will build one super partial graph $H^*$ (at the end of Step 3) which corresponds to one skeleton. Since we have a constant number of such partial graphs (at most $2^{d^2k^2}$), our family $\cgF$ of skeletons will have constant size. Recall that $T$ is connected in $H$, $|T|=k$, and $N_H(T)$ is non-empty otherwise no connected graphs would be consistent with $\tup_k$ and $\otup_k$. Therefore, we can apply the procedure in \Cref{lem:alg_layers} in time $O(n^2)$ to obtain $H'$ and the layering $(L_i)^\ell_{i=0}$. We already know every edge and non-edges in $L_0$, in $L_1$, and between $L_0$ and every other layers. The procedure of \Cref{lem:alg_layers} also determines every edge and non-edges between the remaining layers and guarantees our invariant. Thus, the only unknown edges of $H'$ are inside of each layer $L_i$ for $i\geq 2$. 

For every $2\leq i\leq \ell$, let $\cgV_i$ be the family of maximal sets of vertices in layer $L_i$ with the same neighborhood in the previous layer $L_{i-1}$. Let $\{V_1,\dots,V_m\}=\bigcup^\ell_{i=2}\cgV_i$ be all such sets for all layers. By definition, these sets have size at most the maximum degree of $G$ which is $d$. These sets can be 
 easily determined in time $O(d^2n)$. Let $uv\in E_U(H')$ be an unknown edge that is not inside $V_j$ for any $1\leq j\leq m$. As observed previously, $uv$ must be in a layer $L_i$ for some $2\leq i\leq \ell$. This implies that $u$ has a neighbor $w$ in $L_{i-1}$ that is not adjacent to $v$. Since $L_0=T$ is connected, $|T|\geq k$, and every edge and non-edge in $T$ is known, we can find in $O(k)$ time a connected subgraph $C$ of $H'$ of size $k-3$ in $G[L_0\cup\dots\cup L_{i-2}]$ containing a neighbor of $w$ in $L_{i-2}$. Observe that $v$ is isolated in $G[V(C)\cup\{v,w\}]$. Consider $V(C)\cup\{u,v,w\}$. If it is in $\tup_k$, then $uv$ must be an edge, otherwise, $uv$ must be a non-edge. Therefore, every edge outside of a $V_j$ can be determined in time $O(kn^2)$ in total. We call the resulting partial graph $H^*$. For fixed $k$ and $d$, Step 2 can be done in time $O(n^2)$ for each partial graph $H$.

\textbf{Step 3.} For every $1\leq j\leq m$, we determine every unknown edge $uv$ in $V_j$ such that $u$ (resp. $v$) is contained in a connected subgraph $G'$ (of the current partial graph) of order $k-1$, $v\notin V(G')$ (resp. $u\notin V(G')$), and $v$ (resp. $u$) is not adjacent to any vertex in $G'-u$ (resp. $G'-v$), i.e $vw\in E_N$ (resp. $uw$) for every $w\in V(G'-u)$ (resp. $w\in V(G'-v)$). Indeed, if $V(G'+v)\in \tup_k$ (resp. $V(G'+u)\in\tup_k$), then $uv$ must be an edge, otherwise, $uv$ must be a non-edge. We repeat this procedure until no such unknown edge remains and call the resulting graph $H^*$. There are at most $2^{\binom{d}{2}}$ such unknown edges for each $V_j$ since $|V_j|\leq d$; therefore, there are at most $2^{d^2}m=O(2^{d^2}n)$ such unknown edges. For each edge, it takes $O(k)$ time to verify the wanted property. Thus, it takes $O(2^{d^2}kn)$ time to check every unknown edge once. After each round, either an unknown edge becomes known or we did not find any with the wanted property and Step 3 ends. Hence, this procedure lasts for at most $O(2^{d^2}kn)$ rounds which results in $O(2^{2d^2}k^2n^2)$ time in total. For fixed $k$ and $d$, Step 3 takes $O(n^2)$ time.

Before moving on to the final step, we claim the following.

\begin{claim}
    For every $k$-set $S$, there exists $1\leq j\leq m$, such that $V_j$ contains every important unknown edge of $S$ in $H^*$.
\end{claim}

\begin{proof}
Suppose by contradiction that there exist important unknown edges $uv$ and $xy$ for $S$ in $H^*$ such that $uv$ is in $V_i$ and $xy$ is in $V_j$ with $i\neq j$. W.l.o.g. there must exist a path $P$ in $S$ between $u$ and $x$ on $E(H^*)$ and $E_U(H^*)$, otherwise, $S$ will always induce a disconnected subgraph so $uv$ and $xy$ are not important. 
    
We can assume that the unknown edges of $P$ are not important for $S$, otherwise we can always consider a shorter path of $P$ between two important unknown edges that are in different sets of $V_1,\dots,V_m$ (recall that edges outside of $V_1,\dots,V_m$ are all determined). 
    
Moreover, we can even pick $P$ to be only on edges of $E(H^*)$ thanks to \Cref{prop:unimportant}. Indeed, if $P$ contains some unimportant edge $wz$, then we can replace $wz$ in $P$ by a path in $S$ on edges of $E(H^*)$ between $w$ and $z$ since $K_S-E_N(H)$ cannot be disconnected (otherwise no unknown edge can be important). 
    
Observe that $v$ cannot be adjacent in $H^*$ to any vertex of $P-u$ since it would imply that $uv$ is unimportant for $S$ by \Cref{prop:unimportant}. For every vertex $t$ of $P-u$, we will show that $vt$ has to be a known edge which implies that it must be a non-edge due to the previous observation. Indeed, $vx$ is known after Step 2 since it is outside of $V_1,\dots,V_m$ (recall that they are disjoint and $v\in V_i$, $x\in V_j$). For the internal vertices $t$ of $P$, if $vt$ is an unknown edge, then it must be unimportant for $S$; otherwise there exists a shorter path in $P$ between two important unknown edges $vt$, which must be in $V_i$ since $V_1,\dots,V_m$ are disjoint, and $xy$, which is in $V_j$. However, By \Cref{prop:unimportant}, there exists a path between $v$ and $t$ in $S$ on $E(H^*)$; combined with the path between $u$ and $t$, this forms a path between $u$ and $v$ in $S$ on $E(H^*)$. By \Cref{prop:unimportant}, this implies that $uv$ is unimportant for $S$ in $H^*$, a contradiction. To sum up, for every vertex $t\in P-u$, $vt\in E_N(H^*)$. Symmetrically, the same holds for $y$, i.e. for every vertex $t\in P-x$, $yt\in E_N(H^*)$. 
    
 Now, let $t$ be a ``lowest point'' on $P$, i.e. $t$ belongs to the layer $L_{i'}$ with the smallest $i'$ among vertices in $P$. First, suppose that we can choose $t=x$. There exists a neighbor $q$ of $x$ in $L_{i'-1}$ such that $qv\in E_N(H^*)$. Indeed, $qv$ must be a known edge since an edge between two layers is known after Step 2 and if $qv$ is adjacent for every neighbor $q$ of $x$, then $v$ must be in $V_j$, a contradiction. Now, we take a neighbor $r$ of $q$ in $L_{i'-2}$ which exists since $i'\geq 2$ ($L_0$ and $L_1$ are completely determined). It suffices to find a connected subgraph $C$ in $H^*[L_0,\dots,L_{i'-2}]$ containing $r$ and such that $C$ along with $r$, $q$, $P$, and $v$ has order $k$ (such $C$ exists since $L_0=T$ has size order at least $k$ and layers $0$ to $i'-2$ are connected). Observe that $vq\in E_N(H^*)$ and $vc\in E_N(H^*)$ for every $c\in V(C)$ due to there being at least one layer ($L_{i'-1}$) separating $v$ and these vertices. Such a set verifies the property required to determine unknown edges in Step 3, therefore $uv$ must be known, a contradiction. Symmetrically, the same holds when we can choose $t=u$, i.e. $xy$ must be a known edge. Finally, we look at the case where $t$ is an internal vertex of $P$, i.e. $u$ (resp. $x$) is in layer $L_{j'}$ for some $j'>i$. The same argument as above will imply that both $uv$ and $xy$ must be known after Step 3. Indeed, a neighbor $q$ of $t$ in $L_{i'-1}$ is such that $qv\in E_N(H^*)$ since $v\in L_{j'}$ for $j'>i$ and the rest of the arguments unfolds in a similar fashion.    
\end{proof}

\textbf{Step 4.} To finish our skeletons, we only need to generate the family of collections of graphs $\cgC_1,\dots,\cgC_m$. For every $1\leq j\leq m$, $\cgC_j$ has size at most $2^{\binom{d}{2}}$ since $V_j$ has size at most $d$. For every possible supergraph $G_j$ on $V_j$ that results in vertices with degree at most $d$ in the whole graph (this is well defined since the $V_1,\dots,V_m$ are disjoint), we verify if $G_j$ satisfy our $k$-sets by checking the following. For every $k$-set $S\in\tup_k$ (resp. $S'\in\otup_k$) such that $S\cap V_j\neq \emptyset$ (resp. $S'\cap V_j\neq \emptyset$), we assume that every unknown edge in $S$ (resp. $S'$) that is not included in $V_j$ is a non-edge (resp. an edge), the unknown edges in $V_j$ will be replaced with edges of $G_j$. If $S$ (resp. $S'$) induces a connected (resp. disconnected) subgraph with these assumptions, then we say that $G_j$ satisfies $S$ (resp. $S'$). If $G_j$ satisfies every $k$-set with a non-empty intersection with $V_j$, then we keep $G_j$ as a graph of our collection $\cgC_j$. For this, we need to check every $k$-set in $\tup_k$ and $\otup_k$. Therefore, Step 4 is done in $O(n^k)$ time.

\textbf{Correctness.} Now, we prove that the completions of our family of skeletons correspond to all graphs with maximum degree $d$ that are consistent with $\tup_k$ and $\otup_k$. Due to our invariant, up to the end of Step 3, a graph with maximum degree $d$ that is consistent with $\tup_k$ and $\otup_k$ must be a supergraph of one of our partial graphs. If for every skeleton, $\cgC_1,\dots,\cgC_m$ contain every possible supergraph on $V_1,\dots,V_m$, then a consistent graph must be a completion of a skeleton. In Step 3, we go further by keeping only ``relevant'' supergraphs. Now, we show that a graph with maximum degree $d$ is consistent with $\tup_k$ and $\otup_k$ if and only if it is a completion of one of our skeleton. This implies that, despite our supergraphs being determined independently for each $V_j$, any combination of these graphs will produce a consistent graph.

Suppose by contradiction that there exists a completion $G$ of a skeleton that consists of $H^*,V_1,\dots,V_m,\cgC_1,\dots,\cgC_m$ that is not consistent with $\tup_k$ and $\otup_k$ (since they all have maximum degree at most $d$). First, suppose that there exists $S\in\tup_k$ (resp. $S\in\otup_k$) such that $G[S]$ is disconnected (resp. connected). Unimportant unknown edges for $S$ in $H^*$ cannot change the connectivity of $S$ no matter their assignment; therefore, we can assume that every unimportant unknown edge for $S$ in $H^*$ is set to a non-edge in $G[S]$. The remaining unknown edges for $S$ must be important and by the previous claim, we know these edges are all inside a $V_i$ for some $1\leq i\leq m$. Since graphs in $\cgC_i$ are chosen such that they satisfy all $k$-sets that intersect with $V_i$, $G[S]$ cannot be disconnected; otherwise, no such graphs exist, i.e. no graphs with maximum degree $d$ can be consistent with $\tup_k$ and $\otup_k$. 
This concludes the proof of \Cref{thm:bounded-degree}.\qed
\end{proof}


\section{Uniqueness}\label{sec:uniqueness}



In this section, we consider a graph $G$ with $k$-sets $\tup_k(G)$ and $\otup_k(G)$ and we study the structure of the family of graphs consistent with $\tup_k(G)$ and $\otup_k(G)$ to see when $G$ is uniquely reconstructible.

The following property of connected graphs will be crucial. 

\begin{proposition}\label{prop:conn-cond}
    Let $G=(V,E)$ be a graph and let $uv\in E$. For every $S\in \tup_k(G)$ such that $v\in S,u\notin S$, there exists $v'\in S\setminus \{v\}$ such that $(S\setminus \{v'\})\cup\{u\}\in \tup_k(G)$.
\end{proposition}

\begin{proof}
Since $G[S]$ is connected, there exists a spanning tree of $G[S]$. Since a tree always has at least two leaves, consider $v'\neq v$ to be a leaf of such a spanning tree. Observe that $G[S\setminus\{v'\}]$ also contains a spanning tree as a subgraph. Therefore, it is connected. Combining it with the fact that $uv\in E$, we conclude that $G[(S\setminus \{v'\})\cup\{u\}]$ is connected. In other words, $(S\setminus \{v'\})\cup\{u\}\in \tup_k(G)$.
\qed
\end{proof}

This motivates the following definition.

\begin{definition}
    Let $G=(V,E)$ be a graph and let $uv\notin E$. We say that $u$ and $v$ are \emph{clear non-neighbors} if (i) there exists $S\in \tup_k(G)$ such that $v\in S,u\notin S$ and for every $v'\in S\setminus\{v\}$, it holds that $(S\setminus \{v'\})\cup\{u\}\in \otup_k(G)$, or (ii) there exists $S\in \tup_k(G)$ such that $u\in S,v\notin S$ and for every $u'\in S\setminus\{u\}$, it holds that $(S\setminus \{u'\})\cup\{v\}\in \otup_k(G)$. Otherwise, we say that they are \emph{fake neighbors}.
\end{definition}

\Cref{prop:conn-cond} implies that non-edges between clear non-neighbors are present in every graph consistent with the same $k$-sets. 

\begin{corollary}\label{cor:non-edge-cond}
    Let $G=(V,E)$ be a graph and let $uv\notin E$. If $u$ and $v$ are clear non-neighbors, then for every graph $G'$ consistent with $\tup_k(G)$ and $\otup_k(G)$, we have $uv\notin E(G')$.
\end{corollary}

What follows is a sufficient condition for two vertices to be clear non-neighbors.

\begin{proposition}\label{prop:suff-cond-non-neighbors}
    Let $G=(V,E)$ be a graph and let $u\in V$. If there exists a subgraph $G_u$ such that $u$ is isolated in $G_u$, $G_u-u$ is connected, and $|V(G_u-u)|\geq k$, then for every $v\in V(G_u)$, $u$ and $v$ are clear non-neighbors.
\end{proposition}

\begin{proof}
Consider $v\in V(G_u)$, a connected subgraph of $G_u$ of size $k$ containing $v$, and let $S$ be its set of vertices. We have $S\in\tup_k(G)$, $v\in S$ and $u\notin S$. Moreover, for every vertex $u'$ in $S-\{v\}$, $S-\{u'\}+\{u\}\in\otup_k(G)$ because $u$ is an isolated vertex in $G_u$. Therefore, by definition, $u$ and $v$ are clear non-neighbors.
\qed
\end{proof}

\Cref{lem:struct-nedges-nunique} describes the structure of a graph in which there is some pair of fake neighbors. 

\begin{lemma}\label{lem:struct-nedges-nunique}
    Let $G=(V,E)$ be a graph and let $uv\notin E$ be such that $u$ and $v$ are fake neighbors.
    Let $C_u$ (resp. $C_v$) be the connected component containing $u$ in $G-N(v)$ (resp. $v$ in $G-N(u)$). We have $|V(C_u)|\leq k-1$ and $|V(C_v)|\leq k-1$. In particular, the connected component(s) containing $u$ and $v$ in $G-(N(u)\cap N(v))$ has size (have size in total) at most $|V(C_u)\cup V(C_v)|\leq 2k-2$. Moreover, if $|V(C_u)\cup V(C_v)|\leq k-1$, then $G$ is not uniquely reconstructible.
\end{lemma}





    

\begin{proof}
    First, observe that $v$ (resp. $u$) is not adjacent to any vertex in $C_u$ (resp. $C_v$). Therefore, $v$ is an isolated vertex in $C_u+v$. By~\Cref{prop:suff-cond-non-neighbors}, $|V(C_u)|\leq k-1$. Symmetrically, the same reasoning shows that $|V(C_v)|\leq k-1$. The connected component(s) containing $u$ and $v$ in $G-(N(u)\cap N(v))$ is a subgraph of $G[V(C_u)\cup V(C_v)]$, so it (they) has size (have size in total) at most $|V(C_u)\cup V(C_v)|\leq 2k-2$.

    Now, suppose that $|V(C_u)\cup V(C_v)|\leq k-1$. Let $G'=G+uv$, we will show that $G'$ is consistent with $\tup_k(G)$ and $\otup_k(G)$; in other words, $G$ is not uniquely reconstructible. If $G'$ is not consistent with some $k$-set of $G$, then it implies that there exists $S\in \otup_k(G)$ containing $u$ and $v$ such that $G'[S]$ is connected. Since $G[S]$ is disconnected and $G'=G+uv$, $uv$ must be a bridge in $G'[S]$ and $G[S]$ has exactly two connected components, one containing $u$, the other containing $v$. Therefore, $S$ cannot contain any common neighbor of $u$ and $v$, and $S\subseteq V(C_u)\cup V(C_v)$. This is a contradiction because $k=|S|\leq |V(C_u)\cup V(C_v)|\leq k-1$. This completes the proof.
    \qed
\end{proof}

\subsection{Uniqueness for triangle-free graphs}\label{sec:unique-triangle-free}

We start with showing that if in a triangle-free graph we have a vertex $v$ of large degree, then the neighborhood of $v$ is uniquely reconstructible.

\begin{lemma}\label{lem:k-1-r}
    Let $k$ be an integer with $k\geq 3$ and $r\geq 3k-6$. The graph $K_{1,r}$ is uniquely reconstructible in the class of triangle-free graphs.
\end{lemma}

\begin{proof}
Let $r\geq 3k-6$. Suppose that there exists $G$ that is consistent with $\tup_k(K_{1,r})$ and $\otup_k(K_{1,r})$. Let $V$ be the set of vertices of $K_{1,r}$ and $v$ be the vertex in $K_{1,r}$ with degree $r$. For every $S\subseteq V\setminus \{v\}$ such that $|S|=k-1$, at least one vertex of $S$ must be adjacent to $v$ since $S\cup\{v\}\in\tup_k(K_{1,r})$. Let $A$ be the set of the vertices non-adjacent to $v$. Then $A$ has size at most $k-2$. If $A=\emptyset$, then by the fact that we consider only triangle-free graphs, we already have that $G$ is exactly $K_{1,r}$.

So suppose that $A\neq\emptyset$. Let $B=V\setminus(A\cup\{v\})$. For $a\in A$ and $B'\subseteq B$ such that $|B'|=k-2$, there is at least one vertex in $B'$ adjacent to $a$ since we have $B'\cup \{a,v\}\in\tup_k(K_{1,r})$ and $av\notin E(G)$. Therefore, there are at most $k-3$ vertices in $B$ non-adjacent to $a$. 
Consequently, the degree of any vertex $a\in A$ is at least $|B|-k+3$, and we have that $|B|=r-|A|\geq 3k-6-k+2=2k-4$. Thus $a$ is adjacent to at least $k-1$ vertices of $B$, let $b_1,\ldots,b_{k-1}$ be neighbors of $a$ in $B$. Then $\{a,b_1\ldots,b_{k-1}\}$ induces a connected subgraph in $G$, but $\{a,b_1\ldots,b_{k-1}\}\notin \tup_k(K_{1,r})$ as the only connected $k$-sets are the ones containing $v$, a contradiction.
\qed
\end{proof}




Now, we are ready to prove that triangle-free graphs are uniquely reconstructible.

\begin{theorem}
    Let $k$ be an integer with $k\geq 3$. If $G$ is a triangle-free connected graph on at least $(2k-2)(3k-7)^2+3k-6$ vertices, then $G$ is uniquely reconstructible in the class of triangle-free graphs. 
\end{theorem}

\begin{proof}
    Suppose by that there exists $G'$ that is consistent with $\tup_k(G)$ and $\otup_k(G)$. If an induced subgraph $G[V']$ of $G$ is uniquely reconstructible, then $G'[V']=G[V]$. Therefore, if there is a vertex $v$ of degree at least $3k-6$, then by \Cref{lem:k-1-r}, $G[N[v]]$ is uniquely determined. In this case, $G'$ can be uniquely reconstructed by \Cref{l:k3free_large_deg_graph} since $3k-6\geq 2k-4$. In other words, $G'=G$. Now, we can assume that the maximum degree of $G$ is at most $3k-7$.  
    
    First, we show that if $uv\notin E(G)$, then $uv\notin E(G')$. If $u,v$ are clear non-neighbors, then by \cref{cor:non-edge-cond}, $uv\notin E(G')$. So now assume that $u$ and $v$ are fake neighbors. By \Cref{lem:struct-nedges-nunique}, the connected component(s) of $G-(N(u)\cap N(v))$ containing $u$ and $v$ have size at most $2k-2$. In particular, it implies that $u$ and $v$ have a common neighbor since $|V(G)|>2k-2$. We will show that $u$ and $v$ also have a common neighbor in $G'$, which will prove that $uv\notin E(G')$ since $G'$ must be triangle-free.
    
    Observe that $|N(u)\cap N(v)|\leq 3k-7$ due to the bound on the maximum degree of $G$. Furthermore, the number of connected components of $G-(N(u)\cap N(v))$ is at most $(3k-7)^2$, again by the bound on the maximum degree. Thus, there is at least one connected component $C$ of size at least $2k-1$ since $|V(G)\setminus (N(u)\cap N(v))|\geq (2k-2)(3k-7)^2+3k-6-3k+7=(2k-2)(3k-7)^2+1$. As a result, $C$ cannot contain $u$ or $v$. Moreover, for every vertex $w\in V(C)$, the $u$ (resp. $v$) and $w$ are clear non-neighbors due to \Cref{prop:suff-cond-non-neighbors} since $u$ (resp. $v$) are isolated in $C+u$ (resp. $C+v$). By \Cref{cor:non-edge-cond}, the edges $uw$ and $vw$ are also not in $E(G')$ for every $w\in V(C)$. Since $G$ is connected, there exists $w\in V(C)$ and $z\in N(u)\cap N(v)$ such that $wz\in E(G)$. Now, let us show that $zu$ and $zv$ are also present in $E(G')$.
    
    There exists $S\in\tup_k(G)$ such that $u,z\in S$ and $S-\{u,z\}\subseteq V(C)$. If $zu\notin E(G')$, then $G'[S]$ is disconnected since $uw$ is a non-edge in $G'$ for every $w\in S-\{u,z\}\subseteq V(C)$. This contradicts the fact that $G'$ is consistent with $\tup_k(G)$ and $\otup_k(G)$. Therefore, $zu\in E(G')$ and similarly, the same holds for $zv$. Finally, since $z$ is a common neighbor of $u$ and $v$ in $G'$ and $G'$ is triangle-free, $uv$ is a non-edge in $G'$.

    By symmetry, if $uv\notin E(G')$, then $uv\notin E(G)$ as we can treat $G$ as a graph consistent with $\tup_k(G')$ and $\otup_k(G')$. Therefore $uv\in E(G)$ if and only if $uv\in E(G')$, which completes the proof.
    \qed
\end{proof}

\section{NP-hardness}

In this section we will prove that for $k\geq4$, if we are not given the full set of connected/disconnected $k$-sets, then it is NP-hard to determine, whether there exists a graph consistent with the given $k$-sets. More formally, we define the following problem.

\noindent\fbox{%
    \parbox{\textwidth}{
\textsc{$k$-Reconstruction}

\emph{Input:} Set $V$, sets $\cgS_k,\overline{\cgS}_k\subseteq {V\choose k}$.

\emph{Question:} Does there exists a graph $G$ on $V$ such that for every $S\in \cgS_k$, the graph $G[S]$ is connected, and for every $S'\in \overline{\cgS}_k$, the graph $G[S']$ is disconnected?
}
}

Let us point out that in contrary to previous sections, we do not assume here that $\cgS_k\cup\overline{\cgS}_k={V\choose k}$. We prove the following.

\begin{theorem}
    Let $k\geq 4$. Then the \textsc{$k$-Reconstruction} problem is NP-hard. Moreover, there is no algorithm that solves \textsc{$k$-Reconstruction} on instances $(V,\cgS_k,\overline{\cgS}_k)$ with $n$ vertices and $t=|\cgS_k\cup \overline{\cgS}_k|$ in time $2^{o(n)}\cdot n^{O(1)}$ neither $2^{o(t)}\cdot n^{O(1)}$, unless the ETH fails.
\end{theorem}

\begin{proof}
    We will reduce from $3$-SAT. Let $\phi$ be an instance of $3$-SAT with $n$ variables $x_1,\ldots,x_n$ and $m$ clauses. We can assume that all clauses contain exactly $3$ variables.

    We construct an instance $(V,\cgS_k,\overline{\cgS}_k)$ of \textsc{$k$-Reconstruction} as follows. First, we define the vertex set $V$. For every variable $x_i$, we introduce two vertices $x_i, y_i$ and for $j\in [k-1]$ we introduce $x_i^j$ and $y_i^j$. Moreover, we introduce vertex $v$ and and for all $i\in [k-3],j\in [k-1]$, we introduce vertices $u_i,u_i^j,w_i,w_i^j,v^j$. This completes the definition of the vertex set $V$. Note that $|V|=O(n)$.

    \textbf{Auxiliary partial graph.} We will define an auxiliary partial graph $H$ on $V$. First, for every $i\in [k-3]$ we add edges $u_iv$, and for every $j\in [k-1]$, we add edges $u_iu_i^j,w_iw_i^j,vv^j,x_ix_i^j,y_iy_i^j$. Moreover, we make each $w_i$ adjacent to all $x_j,y_j$ and we introduce all edges $x_ix_j,y_iy_j,x_iy_j$ for $i\neq j$.
    The edges $vx_i,vy_i$ are unknown. All the remaining pairs of vertices are non-edges (see \Cref{fig:hardness}).

\begin{center}
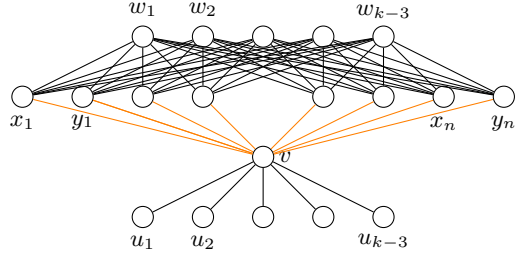
\begin{figure}[h]
    \centering
    \begin{tikzpicture}[every node/.style={draw,circle,fill=white,inner sep=0pt,minimum size=8pt},every loop/.style={},scale=0.8]
    \node[label=right:\footnotesize{$v$}] (v) at (0,0) {};
    \node[label=below:\footnotesize{$u_1$}] (u1) at (-2,-1) {};
    \node[label=below:\footnotesize{$u_2$}] (u2) at (-1,-1) {};
    \node (u3) at (0,-1) {};
    \node (u4) at (1,-1) {};
    \node (u5) at (2,-1) {};
    \node[draw=none,fill=none] (u5') at (2,-1.4) {\footnotesize{$u_{k - 3}$}};
    \draw (v)--(u1);
    \draw (v)--(u2);
    \draw (v)--(u3);
    \draw (v)--(u4);
    \draw (v)--(u5);

    \node[label=below:\footnotesize{$x_1$}] (x1) at (-4,1) {};
    \node[label=below:\footnotesize{$y_1$}] (y1) at (-3,1) {};
    \node (x2) at (-2,1) {};
    \node (y2) at (-1,1) {};
    \node (x3) at (1,1) {};
    \node (y3) at (2,1) {};
    \node[label=below:\footnotesize{$x_n$}] (xn) at (3,1) {};
    \node[label=below:\footnotesize{$y_n$}] (yn) at (4,1) {};

    \draw[color=orange] (v)--(x1);
    \draw[color=orange] (v)--(y1);
    \draw[color=orange] (v)--(x2);
    \draw[color=orange] (v)--(y1);
    \draw[color=orange] (v)--(y2);
    \draw[color=orange] (v)--(x3);
    \draw[color=orange] (v)--(y3);
    \draw[color=orange] (v)--(xn);
    \draw[color=orange] (v)--(yn);

    \node[label=above:\footnotesize{$w_1$}] (w1) at (-2,2) {};
    \node[label=above:\footnotesize{$w_2$}] (w2) at (-1,2) {};
    \node (w3) at (0,2) {};
    \node (w4) at (1,2) {};
    \node (w5) at (2,2) {};
    \node[draw=none,fill=none] (w5') at (2,2.4) {\footnotesize{$w_{k - 3}$}};

    \foreach \k in {1,2,3,4,5}
    {
    \draw (w\k)--(x1);
    \draw (w\k)--(y1);
    \draw (w\k)--(x2);
    \draw (w\k)--(y2);
    \draw (w\k)--(x3);
    \draw (w\k)--(y3);
    \draw (w\k)--(xn);
    \draw (w\k)--(yn);
    }
    
    \end{tikzpicture}
    \caption{\label{fig:hardness}The auxiliary partial graph $H$. Orange edges denote the unknown edges. We ommited here $k-1$ private neighbors of each vertex in the figure. Furthermore, the vertices from $\{x_i,y_j \ | \ i,j\in[n]\}$ induce a clique with the matching $\{x_iy_i \ | \ i\in [n]\}$ removed.}
\end{figure}
\end{center}

    \textbf{Defining connected and disconnected $k$-sets.} We add the following $k$-sets to $\cgS_k,\overline{\cgS}_k$.
    \begin{enumerate}
        \item \textbf{Edges $vu_i$.} Let $G_1:=H[\{v,v^j,u_i,u_i^j \ | \ i\in[k-3],j\in[k-1]\}]$. Note that $G_1$ is a graph, since we do not have unknown edges inside the vertex set. We add all $k$-sets of $\tup_k(G_1)$ to $\cgS_k$ and all $k$-sets of $\otup_k(G_1)$ to $\overline{\cgS}_k$.
        \item \textbf{Non-edges $vw_i$.} Let $G_2:=H[\{v,v^j,w_i,w_i^j \ | \ i\in[k-3],j\in [k-1]\}]$, again $G_2$ is a graph. We add all $k$-sets of $\tup_k(G_2)$ to $\cgS_k$ and all $k$-sets of $\otup_k(G_2)$ to $\overline{\cgS}_k$.
        \item \textbf{Non-edges $x_iu_j, y_iu_j$.} Let $i\in [n]$. We define $G^i=H[x_i,x_i^j,y_i,y_i^j,u_\ell,u_\ell^j \ | \ j\in[k-1],\ell\in[k-3]\}]$, and $G^i$ is a graph. We add all $k$-sets of $\tup_k(G^i)$ to $\cgS_k$ and all $k$-sets of $\otup_k(G^i)$ to $\overline{\cgS}_k$.
        \item \textbf{Variables.} For every $i\in [n]$ we add the set $\{x_i,y_i,v,u_1,\ldots,u_{k-3}\}$ to $\overline{\cgS}_k$.
        \item \textbf{Clauses.} For every clause with literals $\ell_1,\ell_2,\ell_3$, for $j\in[3]$, let $z_j=x_i$ if $\ell_i=x_i$ and $z_j=y_i$ if $\ell_j=\neg x_i$. We add the set $\{z_1,z_2,z_3,v,u_1,\ldots,u_{k-4}\}$ to $S_k$. 
    \end{enumerate}

    This completes the definition of $\cgS_k$ and $\overline{\cgS}_k$. Let us verify the equivalence of the instances. 

    First assume that there is a graph $G$ on $V$ that satisfies all the $k$-sets from $\cgS_k,\overline{\cgS}_k$, i.e., for every $S\in \cgS_k$ the graph $G[S]$ is connected and for every $S'\in \overline{\cgS}_k$, the graph $G[S']$ is disconnected. Before we define the truth assignment for variables of $\phi$, let us discuss some properties of $G$.

   \begin{claim}
       Let $i\in [k-3]$. Then we have that $u_iv\in E(G)$ and $w_iv\notin E(G)$.
   \end{claim}
   \begin{proof}
       Suppose first that $u_iv\notin E(G)$. Since $\{v,u_i,u_i^1,\ldots,u_i^{k-2}\}\in \cgS_k$, then there is some $u_i^j$ adjacent to $v$, without loss of generality $u_i^1v\in E(G)$. We also know that $\{u_i,u_i^1,\ldots,u_i^{k-1}\}$ induces a connected subgraph of $G$. There is a vertex $u_i^j\neq u_i^1$, such that the set $\{u_i,u_i^1,\ldots,u_i^{k-1}\}\setminus \{u_i^j\}$ still induces a connected subgraph (we can take a spanning tree and remove a leaf which is not $u_i^1$). Thus the graph induced by $\{v,u_i^1,\ldots,u_i^{k-1}\}\setminus \{u_i^j\}$ is connected, which contradicts the fact that we added this $k$-set to $\overline{\cgS}_k$ in 1.

       Now suppose that $w_iv\in E(G)$. Since $\{v,v^1,\ldots,v^{k-1}\}\in \cgS_k$, there is some $v^j$, say $v^{k-1}$, such that the graph induced by $\{v,v^1,\ldots,v^{k-2}\}$ is still connected, and thus the graph induced by $\{w_i,v,v^1,\ldots,v^{k-2}\}$ is connected, a contradiction.
   \end{proof}

   \begin{center}
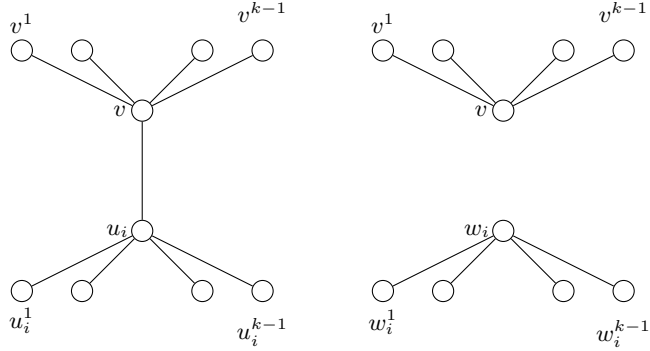
\begin{figure}[h]
    \centering
    \begin{tikzpicture}[every node/.style={draw,circle,fill=white,inner sep=0pt,minimum size=8pt},every loop/.style={},scale=0.8]
    \node[label=left:\footnotesize{$v$}] (v) at (0,0) {};
    \node[label=left:\footnotesize{$u_i$}] (u) at (0,-2) {};
    \node[label=below:\footnotesize{$u_i^1$}] (u1) at (-2,-3) {};
    \node (u2) at (-1,-3) {};
    \node (u3) at (1,-3) {};
    \node[label=below:\footnotesize{$u_i^{k-1}$}] (u4) at (2,-3) {};
    \node[label=above:\footnotesize{$v^{1}$}] (v1) at (-2,1) {};
    \node (v2) at (-1,1) {};
    \node (v3) at (1,1) {};
    \node[label=above:\footnotesize{$v^{k-1}$}] (v4) at (2,1) {};
    \draw (v)--(u);
    \foreach \k in {1,2,3,4}
    {
    \draw (v)--(v\k);
    \draw (u)--(u\k);
    }


    \node[label=left:\footnotesize{$v$}] (v) at (6,0) {};
    \node[label=left:\footnotesize{$w_i$}] (w) at (6,-2) {};
    \node[label=below:\footnotesize{$w_i^1$}] (w1) at (4,-3) {};
    \node (w2) at (5,-3) {};
    \node (w3) at (7,-3) {};
    \node[label=below:\footnotesize{$w_i^{k-1}$}] (w4) at (8,-3) {};
    \node[label=above:\footnotesize{$v^1$}] (v1) at (4,1) {};
    \node (v2) at (5,1) {};
    \node (v3) at (7,1) {};
    \node[label=above:\footnotesize{$v^{k-1}$}] (v4) at (8,1) {};
    \foreach \k in {1,2,3,4}
    {
    \draw (v)--(v\k);
    \draw (w)--(w\k);
    }
    
    \end{tikzpicture}
    \caption{\label{fig:hardness2}The auxiliary partial graph $H$ on vertices $v,v^j,u_i,u_i^j$ (left) and on $v,v^j,w_i,w_i^j$ (right).}
    \end{figure}
    \end{center}

Similarly, because of the $k$-sets added in 3., we can obtain the following.

    \begin{claim}
       Let $i\in [n], j\in [k-3]$. Then it holds that $x_iu_j\notin E(G)$ and $y_iu_j\notin E(G)$.
   \end{claim}

We also show that at most one of edges $x_iv,y_iv$ can be present in $G$.

   \begin{claim}
       Let $i\in [n]$. It holds at most one: $x_iv\in E(G)$ or $y_iv\in E(G)$.
   \end{claim}
   \begin{proof}
       Suppose that both $x_i,y_i$ are adjacent to $v$ in $G$. Recall that we also have that $u_jv\in E(G)$, for every $j\in [k-3]$. Therefore, the graph induced by the set $\{x_i,y_i,v,u_1,\ldots,u_{k-3}\}$ is connected, which contradicts the fact that we added this set to $\overline{\cgS}_k$ in 4.
   \end{proof}

   Now we are ready to define a truth assignment $\varphi$ of the variables of $\phi$. For $i\in [n]$, we set $\varphi(x_i)=1$ if $x_iv\in E(G)$, and $\varphi(x_i)=0$ if $y_iv\in E(G)$. If none of the edges is present in $G$, we set $\varphi(x_i)$ arbitrarily.

   Suppose there is some clause $C$ that is not satisfied by $\varphi$, let $\ell_1,\ell_2,\ell_3$ be its literals. For $j\in [3]$, let $z_j=x_i$ if $\ell_i=x_i$ and $z_j=y_i$ if $\ell_j=\neg x_i$. Recall that we added the set $\{z_1,z_2,z_3,v,u_1,\ldots,u_{k-4}\}$ to $\cgS_k$ in 5. Therefore, the graph induced by this set in $G$ must be connected. As we already observed, there are no edges between $z_1,z_2,z_3$ and $u_1,\ldots,u_{k-3}$. Thus at least one of $z_j$ must be adjacent to $v$. By the definition of $\varphi$ the literal $\ell_j$ is true and satisfies $C$, a contradiction.

   It remains to show that if there is a satisfying assignment $\varphi$ of $\phi$, then there is also a graph $G$ on $V$ that satisfies all tuples. We construct $G$ so that we take the partial graph $H$, and for the unknown edges, which are of type $x_iv,y_iv$, we decide $x_iv\in E(G),y_iv\notin E(G)$ if $\varphi(x_i)=1$, and $x_iv\notin E(G),y_iv\in E(G)$ otherwise. Let us verify that $G$ satisfies all $k$-sets from $\cgS_k,\overline{\cgS}_k$.

   Note that the $k$-sets added in 1.--3. do not contain any unknown edges of $H$ and by their definition they are clearly satisfied by $G$. Consider a $k$-set introduced in 4., so a set $\{x_i,y_i,v,u_1,\ldots,u_{k-3}\}$ added to $\overline{\cgS}_k$. By the definition of $G$, there are no edges between $x_i,y_i$ and $u_1,\ldots,u_{k-3}$. Moreover, we added only one of $x_iv,y_iv$ to the edge set. Since $x_iy_i\notin E(G)$, then one of $x_i,y_i$ is isolated, and the graph induced by $\{x_i,y_i,v,u_1,\ldots,u_{k-3}\}$ is indeed disconnected.

   It remains to verify the tuples added in 5. Let $\{z_1,z_2,z_3,v,u_1,\ldots,u_{k-4}\}$ be the set added to $\cgS_k$ for clause $C$ with literals $\ell_1,\ell_2,\ell_3$ such that for $j\in[3]$, $z_j=x_i$ if $\ell_i=x_i$ and $z_j=y_i$ if $\ell_j=\neg x_i$. Since literals $\ell_1,\ell_2,\ell_3$ correspond to different variables, we have that $z_1,z_2,z_3$ form a triangle. Moreover, $v$ is adjacent to all $u_1,\ldots,u_{k-4}$. Since the clause $C$ is satisfied, there is at least one literal $\ell_j$ set to true, and thus $vz_j\in E(G)$. Thus the graph induced by $\{z_1,z_2,z_3,v,u_1,\ldots,u_{k-4}\}$ is connected, which completes the proof.
   \qed
\end{proof}



\textbf{Aknowledgements.} This work started at the 2023 Structural Graph Theory workshop held at the IMPAN B\k{e}dlewo Conference Center in Poland in September 2023. This workshop was a part of STRUG: Stuctural Graph Theory Bootcamp, funded by the ``Excellence initiative - research university (2020-2026)'' of the University of Warsaw.

\bibliographystyle{plain}
\bibliography{biblio}

\begin{thebibliography}{1}

\bibitem{BCEGKMV2023}
Paul Bastide, Linda Cook, Jeff Erickson, Carla Groenland, Marc~van Kreveld,
  Isja Mannens, and Jordi~L. Vermeulen.
\newblock Reconstructing graphs from connected triples.
\newblock In Dani{\"e}l Paulusma and Bernard Ries, editors, {\em
  Graph-Theoretic Concepts in Computer Science}, pages 16--29, Cham, 2023.
  Springer Nature Switzerland.

\bibitem{KMZ2018}
Sampath Kannan, Claire Mathieu, and Hang Zhou.
\newblock Graph reconstruction and verification.
\newblock {\em ACM Trans. Algorithms}, 14(4), aug 2018.

\bibitem{Kelly1942}
P.J. Kelly.
\newblock {\em On isometric transformations}.
\newblock PhD thesis, University of Wisconsin, 1942.

\bibitem{Qi2023}
Yaxin Qi.
\newblock Graph reconstruction from connected triples, 2023.

\bibitem{Ulam1960}
S.M. Ulam.
\newblock {\em A Collection of Mathematical Problems}.
\newblock Interscience tracts in pure and applied mathematics. Interscience
  Publishers, 1960.

\end{thebibliography}

\end{document}